\documentclass{amsart}

\title{Incomparable $\omega_1$-like models of set theory}

\author[Fuchs]{Gunter Fuchs}
\address[G.~Fuchs]{Mathematics,
          The Graduate Center of The City University of New York,
          365 Fifth Avenue, New York, NY 10016
          \&
          Mathematics,
          College of Staten Island of CUNY,
          Staten Island, NY 10314}
\email{Gunter.Fuchs@csi.cuny.edu}
\urladdr{}

\author[Gitman]{Victoria Gitman}
\address[V.~Gitman]{Mathematics,
          The Graduate Center of The City University of New York,
          365 Fifth Avenue, New York, NY 10016}
\email{vgitman@nylogic.org}
\urladdr{http://boolesrings.org/gitman}

\author[Hamkins]{Joel David Hamkins}
 \address[J.~D.~Hamkins]{Mathematics, Philosophy, Computer Science,
          The Graduate Center of The City University of New York,
          365 Fifth Avenue, New York, NY 10016
          \&
          Mathematics,
          College of Staten Island of CUNY,
          Staten Island, NY 10314}
\email{jhamkins@gc.cuny.edu}
\urladdr{http://jdh.hamkins.org}

\thanks{The research of the third author has been supported in part by NSF grant DMS-0800762, PSC-CUNY grant 64732-00-42, CUNY Collaborative Incentive Award 80209-06 20 and Simons Foundation grant 209252. The authors are thankful to Roman Kossak and Ali Enayat for conversations concerning the ideas in this article.  Commentary concerning this paper can be made at \href{http://jdh.hamkins.org/incomparable-omega-one-like-models-of-set-theory}{http://jdh.hamkins.org/incomparable-omega-one-like-models-of-set-theory}.}

\usepackage[hidelinks]{hyperref}
\usepackage{tikz} 
\usetikzlibrary{arrows}
\usepackage{graphicx}
\usepackage{latexsym}
\usepackage{amssymb, latexsym}
\usepackage{amsmath}
\usepackage{mathrsfs}
\usepackage{amsxtra}
\usepackage{diagrams}
\usepackage{amsthm}
\usepackage{url}
\newtheorem{theorem}{Theorem}

\newtheorem{sublemma}{Lemma}[theorem]
\newtheorem{lemma}[theorem]{Lemma}

\theoremstyle{definition}

\newtheorem{question}[theorem]{Question}

\newcommand{\concat}{%
  \mathord{
    \mathchoice
    {\raisebox{1ex}{\scalebox{.7}{$\frown$}}}
    {\raisebox{1ex}{\scalebox{.7}{$\frown$}}}
    {\raisebox{.7ex}{\scalebox{.5}{$\frown$}}}
    {\raisebox{.7ex}{\scalebox{.5}{$\frown$}}}
  }
}
\newcommand{\Coll}{\mathop{\rm Coll}}

\newcommand{\image}{\mathbin{\hbox{\tt\char'42}}}

\newcommand{\Union}{\bigcup}

\newcommand{\of}{\subseteq}
\newcommand{\lt}[1]{{\smalllt}#1}

\newcommand{\smallleq}{\mathrel{\mathchoice{\raise2pt\hbox{$\scriptstyle\leq$}}{\raise1pt\hbox{$\scriptstyle\leq$}}{\raise1pt\hbox{$\scriptscriptstyle\leq$}}{\scriptscriptstyle\leq}}}
\newcommand{\smalllt}{\mathrel{\mathchoice{\raise2pt\hbox{$\scriptstyle<$}}{\raise1pt\hbox{$\scriptstyle<$}}{\raise0pt\hbox{$\scriptscriptstyle<$}}{\scriptscriptstyle<}}}
\newcommand{\Add}{\mathop{\rm Add}}

\newcommand{\ZFC}{{\rm ZFC}}
\newcommand{\Levy}{L{\'e}vy}

\newcommand{\la}{\langle}
\newcommand{\ra}{\rangle}

\newcommand{\Los}{\L o\'s}
\newcommand{\Godel}{G\"{o}del}

\newcommand{\restrict}{\upharpoonright}
\newcommand{\GBC}{{\rm GBC}}
\newcommand{\PA}{{\rm PA}}

\newarrow{Corresponds}<--->
\newarrow{Dashto}{}{dash}{}{dash}>

\newcommand{\ZF}{\ensuremath{\mathrm{ZF}}}

\newcommand{\df}{\it}
\def\<#1>{\langle#1\rangle}
\newcommand\HF{{\rm HF}}
\newcommand\Ord{{\rm Ord}}
\newcommand\satisfies{\models}
\newcommand\elesub{\prec}
\newcommand{\Lowenheim}{L\"owenheim}
\newcommand{\set}[1]{\{\,{#1}\,\}}
\newcommand{\intersect}{\cap}
\newcommand\sbar{{\bar s}}
\newcommand\tbar{{\bar t}}
\newcommand\Q{{\mathbb{Q}}}
\newcommand\PP{{\mathbb{P}}}
\newcommand{\Con}{\mathop{{\rm Con}}}

\begin{document}

\today

\begin{abstract}
We show that the analogues of the Hamkins embedding theorems~\cite{Hamkins2013:EveryCountableModelOfSetTheoryEmbedsIntoItsOwnL}, proved for the countable models of set theory, do not hold when extended to the uncountable realm of $\omega_1$-like models of set theory. Specifically, under the $\diamondsuit$ hypothesis and suitable consistency assumptions, we show that there is a family of $2^{\omega_1}$ many $\omega_1$-like models of $\ZFC$, all with the same ordinals, that are pairwise incomparable under embeddability; there can be a transitive $\omega_1$-like model of \ZFC\ that does not embed into its own constructible universe; and there can be an $\omega_1$-like model of \PA\ whose structure of hereditarily finite sets is not universal for the $\omega_1$-like models of set theory.
\end{abstract}

\maketitle

\section{Introduction}

We should like to consider the question of whether the embedding theorems of Hamkins~\cite{Hamkins2013:EveryCountableModelOfSetTheoryEmbedsIntoItsOwnL}, recently proved for the countable models of set theory, might extend to the realm of uncountable models. Specifically, Hamkins proved that (1) any two countable models of set theory are comparable by embeddability; indeed, (2) one countable model of set theory embeds into another just in case the ordinals of the first order-embed into the ordinals of the second; consequently, (3) every countable model of set theory embeds into its own constructible universe; and furthermore, (4) every countable model of set theory embeds into the hereditarily finite sets $\<\HF,{\in}>^M$ of any nonstandard model of arithmetic $M\satisfies\PA$. The question we consider here is, do the analogous results hold for uncountable models? Our answer is that they do not. Indeed, we shall prove that the corresponding statements do not hold even in the special case of $\omega_1$-like models of set theory, which otherwise among uncountable models often exhibit a special affinity with the countable models. Specifically, we shall construct large families of pairwise incomparable $\omega_1$-like models of set theory, even though they all have the same ordinals; we shall construct $\omega_1$-like models of set theory that do not embed into their own $L$; and we shall construct $\omega_1$-like models of \PA\ that are not universal for all $\omega_1$-like models of set theory.

The Hamkins embedding theorems are  expressed collectively in theorem~\ref{Theorem.HamkinsEmbeddingTheorems} below. An {\df embedding} of one model $\<M,{\in^M}>$ of set theory into another $\<N,{\in^N}>$ is simply a function $j:M\to N$ for which $x\in^My\longleftrightarrow j(x)\in^Nj(y)$, for all $x,y\in M$, and in this case we say that $\<M,{\in^M}>$ {\df embeds} into $\<N,{\in^N}>$; note by extensionality that every embedding is injective.
\begin{figure}\label{Figure.Embedding}
\begin{tikzpicture}[scale=.15,xscale=.8,>=latex]
 \draw[thick] (-26,0) --(-30,11) --(-22,11) --(-26,0);
 \draw[thick] (0,0) --(6,12) --(-6,12) --(0,0);
 \draw (0,2.5) --(1,5.5) --(-1,5.5) --(0,2.5);
 \draw (1.3,6.4) --(2,8.5) --(-2,8.5) --(-1.3,6.4) --(1.3,6.4);
 \draw (2.25,9.25) --(2.7,10.9) --(-2.7,10.9) --(-2.25,9.25) --(2.25,9.25);
 \draw[->] (-24,5.5) to [out=30,in=160] (-1.7,7.6);
 \node at (-16,7) {$j$};
 \node[above] at (0,12) {$N$};
 \node[above] at (-26,11) {$M$};
\end{tikzpicture}
\qquad\quad\raise 25pt\hbox{$x\in^M y\ \longleftrightarrow\ j(x)\in^N j(y)$}
\caption{An embedding $j:M\to N$}
\end{figure}
Thus, an embedding is simply an isomorphism of $\<M,{\in^M}>$ with its range, which is a submodel of $\<N,{\in^N}>$, as illustrated in figure~\ref{Figure.Embedding}. Although this is the usual model-theoretic embedding concept for relational structures, the reader should note that it is a considerably weaker embedding concept than commonly encountered in set theory, because this kind of embedding need not be elementary nor even $\Delta_0$-elementary, although clearly every embedding as just defined is elementary at least for quantifier-free assertions. So we caution the reader not to assume a greater degree of elementarity beyond quantifier-free elementarity for the embeddings appearing in this paper, except where we explicitly remark on it.

\begin{theorem}[Hamkins~\cite{Hamkins2013:EveryCountableModelOfSetTheoryEmbedsIntoItsOwnL}]\label{Theorem.HamkinsEmbeddingTheorems}\
\begin{enumerate}
 \item For any two countable models of set theory $\la M,\in^M\ra$ and $\la N,\in^N\ra$, one of them embeds into the other.
 \item Indeed, such an $\<M,{\in^M}>$ embeds into $\<N,{\in^N}>$ if and only if the ordinals of $M$ order-embed into the ordinals of $N$.
 \item Consequently, every countable model $\la M,\in^M\ra$ of set theory embeds into its own constructible universe $\la L^M,\in^M\ra$.
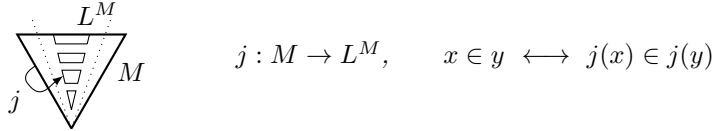
$$\qquad\begin{tikzpicture}[xscale=.06,yscale=.25,>=latex]
 \draw[thick] (-0,0) --(12,5) --(-12,5) --(0,0);
 \draw[dotted] (0,0) --(9,6);
 \draw[dotted] (0,0) --(-9,6);
 \node[anchor=south west] at (-1,5) {$L^M$};
 \draw (0,1) --(1,2) --(-1,2) --(0,1);
 \draw (1.4,2.4) --(2.1,3.1) --(-2.1,3.1) --(-1.4,2.4) --(1.4,2.4);
 \draw (2.5,3.5) --(3,4) --(-3,4) --(-2.5,3.5) --(2.5,3.5);
 \draw (3.5,4.5) --(4,5) --(-4,5) --(-3.5,4.5) --(3.5,4.5);
 \draw[->] (-8,3.33) to [out=190,in=150] (-10,2.5) to [out=-20,in=190] (-1.8,2.8);
 \node[anchor=north east] at (-9,2.5) {$j$};
 \node[anchor=north west] at (8,4) {$M$};
\end{tikzpicture}
\qquad\quad\raise 25pt\hbox{$j:M\to L^M$,\qquad $x\in y\ \longleftrightarrow\ j(x)\in j(y)$}
$$

 \item Furthermore, every countable model of set theory embeds into the hereditary finite sets $\<\HF,{\in}>^M$ of any nonstandard model of arithmetic $M\satisfies\PA$. Indeed, $\HF^M$ is universal for all countable acyclic binary relations.
\end{enumerate}
\end{theorem}

One can begin to get an appreciation for the difference in embedding concepts by observing that \ZFC\ proves that there is a nontrivial embedding $j:V\to V$, namely, the embedding recursively defined as follows $$j(y)=\bigl\{\ j(x)\ \mid\ x\in y\ \bigr\}\cup\bigl\{\{\emptyset,y\}\bigr\}.$$
We leave it as a fun exercise to verify that $x\in y\longleftrightarrow j(x)\in j(y)$ for the embedding $j$ defined by this recursion.\footnote{See~\cite{Hamkins2013:EveryCountableModelOfSetTheoryEmbedsIntoItsOwnL}; but to give a hint here for the impatient, note that every $j(y)$ is nonempty and also $\emptyset\notin j(y)$; it follows that inside $j(y)$ we may identify the pair $\{\emptyset,y\}\in j(y)$; it follows that $j$ is injective and furthermore, the only way to have $j(x)\in j(y)$ is from $x\in y$.} Contrast this situation with the well-known Kunen inconsistency~\cite{Kunen1971:ElementaryEmbeddingsAndInfinitaryCombinatorics}, which asserts that there can be no nontrivial $\Sigma_1$-elementary embedding $j:V\to V$. Similarly, the same recursive definition applied in $L$ leads to nontrivial embeddings $j:L\to L$, regardless of whether $0^\sharp$ exists. But again, the point is that embeddings are not necessarily even $\Delta_0$-elementary, and the familiar equivalence of the existence of $0^\sharp$ with a nontrivial ``embedding'' $j:L\to L$ actually requires a $\Delta_0$-elementary embedding.

We find it interesting to note in contrast to theorem~\ref{Theorem.HamkinsEmbeddingTheorems} that there is no such embedding phenomenon in the the context of the countable models of Peano arithmetic (where an embedding of models of arithmetic is a function preserving all atomic formulas in the language of arithmetic). Perhaps the main reason for this is that embeddings between models of \PA\ are automatically $\Delta_0$-elementary, as a consequence of the MRDP theorem, whereas this is not true for models of set theory, as the example above of the recursively defined embedding $j:V\to V$ shows, since this is an embedding, but it is not $\Delta_0$-elementary, in light of $j(\emptyset)\neq\emptyset$. For countable models of arithmetic $M,N\satisfies\PA$, one can show that there is an embedding $j:M\to N$ if and only if $N$ satisfies the $\Sigma_1$-theory of $M$ and the standard system of $M$ is contained in the standard system of $N$. It follows that there are many instances of incomparability. Meanwhile, it is a consequence of theorem~\ref{Theorem.HamkinsEmbeddingTheorems} statement (4) that the embedding phenomenon recurs with the countable models of finite set theory $\ZFC^{\neg\infty}$, that is, with $\<\HF,{\in}>^M$ for $M\satisfies\PA$, since all nonstandard such models are universal for all countable acyclic binary relations, and so in the context of countable models of $\ZFC^{\neg\infty}$ there are precisely two bi-embeddability classes, namely, the standard model, which is initial, and the nonstandard countable models, which are universal.

Our main theorems are as follows.

\newtheorem*{maintheorems}{Main Theorems}
\begin{maintheorems}\
 \begin{enumerate}
  \item If $\diamondsuit$ holds and \ZFC\ is consistent, then there is a family $\mathcal C$ of $2^{\omega_1}$ many pairwise incomparable $\omega_1$-like models of $\ZFC$, meaning that there is no embedding between any two distinct models in $\mathcal C$.
  \item The models in statement (1) can be constructed so that their ordinals order-embed into each other and indeed, so that the ordinals of each model is a universal $\omega_1$-like linear order. If \ZFC\ has an $\omega$-model, then the models of statement (1) can be constructed so as to have precisely the same ordinals.
  \item If $\diamondsuit$ holds and \ZFC\ is consistent, then there is an $\omega_1$-like model $M\models\ZFC$ and an $\omega_1$-like model $N\models\PA$ such that $M$ does not embed into $\<\HF,{\in}>^N$.
  \item If there is a Mahlo cardinal, then in a forcing extension of $L$, there is a transitive $\omega_1$-like model $M\of\ZFC$ that does not embed into its own constructible universe $L^M$.
 \end{enumerate}
\end{maintheorems}

These results appear later as theorems~\ref{Theorem.FamilyOfIncomparableModels},~\ref{Theorem.FamilyOfIncomparableModelsSameOrdinals},~\ref{Theorem.DiamondImpliesOmega1LikeZFCnotembedPA}, and~\ref{th:incomparableTransitive}. Note that the size of the family $\mathcal C$ in statement (1) is as large as it could possibly be, given that any two elements in a pairwise incomparable family of structures must be non-isomorphic and there are at most $2^{\omega_1}$ many isomorphism types of $\omega_1$-like models of set theory or indeed of structures of size $\omega_1$ in any first-order finite language. Statement (2) shows that the models of the family $\mathcal C$ serve as $\omega_1$-like counterexamples to the assertion that one model of set theory embeds into another whenever the ordinals of the first order-embed into the ordinals of the second.

\goodbreak
\section{$\omega_1$-like models of set theory and other background}
\label{sec:omega1models}

The ordinal $\omega_1$ is the only uncountable ordinal all of whose proper initial segments are countable. Generalizing this, a linear order is {\df $\omega_1$-like}, if it is uncountable, but all proper initial segments are countable. For example, a model of \PA\ is $\omega_1$-like, if it is uncountable, but all proper initial segments are countable. Similarly, a model of set theory $\<M,{\in^M}>$ is {\df $\omega_1$-like}, if it is uncountable, but every rank initial segment $V_\alpha^M$ for $\alpha\in\Ord^M$ is countable. For models of \ZF, this is equivalent to saying that $M$ is uncountable, but every object $y\in M$ has only countably many $\in^M$-predecessors, that is, $\{ x\in M\mid x\in^M y\}$ is countable; for models of \ZFC, it is also equivalent to asserting that the ordinals $\Ord^M$ are $\omega_1$-like as a linear order. The $\omega_1$-like models constitute a gateway from the realm of countable models to the uncountable, sharing and blending many of the features of both kinds of models, and they have been extensively studied both in the case of models of arithmetic and of models of set theory~\cite{Kaufmann1983:BluntAndToplessEndExtensionsOfModelsOfSetTheory, Kaufmann1977:ARatherClasslessModel, Kossak1985:RecursivelySaturatedOmega1LikeModels, Enayat1984:OnCertainElementaryEndExtensionsOfModelsOfSetTheory, MarkerSchmerlSteinhorn:UncountableRealClosedFieldsWithPAIntegerParts}.

One obvious way to construct an $\omega_1$-like model $M$ is as the union of a continuous elementary chain of countable models:
$$
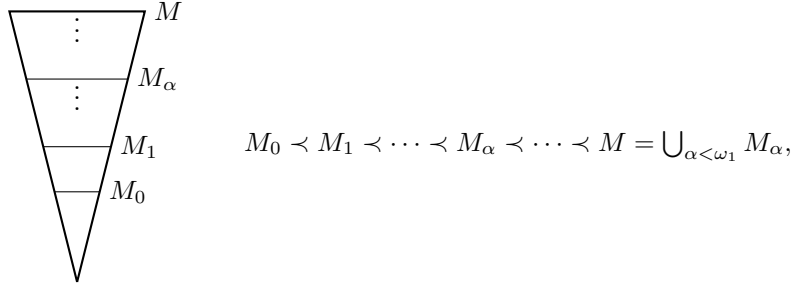
\begin{tikzpicture}[scale=.3,xscale=.5]
\draw[thick] (0,0) -- (6,12) -- (-6,12) -- (0,0);
\draw (-2,4) -- (2,4);
\node at (2,4) [right] {$M_0$};
\draw (-3,6) -- (3,6);
\node at (3,6) [right] {$M_1$};
\draw (-4.5,9) -- (4.5,9);
\node at (4.5,9) [right] {$M_\alpha$};
\node at (0,8.5)  {$\vdots$};
\node at (6,12) [right] {$M$};
\node at (0,11.5) {$\vdots$};
\end{tikzpicture}
\qquad\raise 5em\hbox{$M_0\elesub M_1\elesub\cdots\elesub M_\alpha\elesub\cdots\elesub M=\Union_{\alpha<\omega_1}M_\alpha,$}
$$
At each step we should have an elementary top-extension $M_\alpha\elesub_t M_{\alpha+1}$, meaning that the new elements of $M_{\alpha+1}$ have rank exceeding that of any element of $M_\alpha$, as defined precisely below; and at limit stages $\lambda$ we take unions $M_\lambda=\Union_{\alpha<\lambda}M_\alpha$. It is a consequence of lemma~\ref{Lemma.KeislerMorleyTopExtensions}, a result due to Kiesler and Morley~\cite{KeislerMorley1968:ElementaryExtensionsOfModelsOfSetTheory}, that every countable model of set theory has such an elementary top-extension. In this way, every proper initial segment of the final model $M$ is contained in some $M_\alpha$, which is countable. Thus, the $\omega_1$-like model $M$ grows from the bottom out of its countable elementary initial segments.

Conversely, however, it is not hard to see that every $\omega_1$-like model $M$ must arise exactly in this way as the union of a continuous elementary chain of countable elementary initial segments, because a simple \Lowenheim-Skolem argument shows that there will be unboundedly many such countable elementary initial segments. So the obvious construction method is in a sense the only construction method for building $\omega_1$-like models. Because every $\omega_1$-like model is thus the union of an elementary chain of length $\omega_1$, these models naturally inherit much of the set-theoretic structure and context of $\omega_1$, such as clubs, stationary sets, and constructions via $\diamondsuit$, and it is by taking advantage of this set-theoretic structure that we shall prove our main theorems.

Suppose that $\<M,{\in^M}>$ and $\<N,{\in^N}>$ are models of set theory. We say that the first is a {\df submodel} of the second, written $M\of N$ or more properly $\<M,{\in^M}>\of\<N,{\in^N}>$, if $M$ is a subset of $N$ and $\in^M$ is the restriction of $\in^N$ to the domain $M$, which is to say that the two models agree on whether $a\in b$ for any elements $a,b\in M$. The extension is a {\df transitive} extension, written $M\sqsubseteq N$, if the larger model adds no new elements to old sets, which is to say, $a\in^N b\in M$ implies $a\in M$; this is also sometimes called an {\df end-extension} (this is not the same as a top-extension). Thus, a transitive extension occurs when the submodel is transitive with respect to the membership relation of the larger model, such as in the case of a forcing extension $M\of M[G]$ or of the inner model $L^M\of M$. A {\df top-extension}, in contrast, written $M\sqsubseteq_t N$, occurs when the new sets of the larger model all have higher rank in the von Neumann hierarchy than any old set; that is, if whenever $a\in N\setminus M$ and $b\in M$, then the rank of $a$ in $N$ is higher than the rank of $b$ in $N$. For example, every model $M\satisfies\ZF$ is a top-extension of its rank initial segments $V_\alpha^M\sqsubseteq_t M$. An {\df elementary top-extension}, written $M\elesub_t N$, occurs when a top-extension is also elementary, meaning that every first-order assertion about some objects in $M$ has the same truth value in $M$ as it does in $N$. For models of \ZF, an elementary transitive extension (or elementary end-extension) is the same thing as an elementary top-extension, because $V_\alpha^M$ is definable in $M$ from $\alpha$ and so must by elementarity be equal to $V_\alpha^N$ as defined in $N$. Meanwhile, the nontrivial forcing extensions $M\of M[G]$ provide examples of transitive extensions (end-extensions) that are not top-extensions, and they are never elementary. A model $M$ is {\df topless} in a top-extension $M\sqsubseteq_t N$, if there is no least upper bound of $\Ord^M$ in $N$; otherwise $M$ is {\df topped} in $N$. For any model of set theory $\<M,{\in^M}>$ and any element $a\in M$, let us introduce the following notation
$$a^M=\set{b\in M\mid M\satisfies b\in a},$$
to refer to the set of objects in $M$ that $M$ believes to be elements of $a$. If $M\of N$ is a submodel of another model, then $a^N\intersect M$ is the {\df trace} of $a$ on $M$. Note that another way to say that an extension $M\of N$ is transitive is to say that $a^M=a^N$ for all $a\in M$.

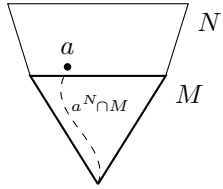
\begin{figure}[here]
\begin{tikzpicture}[scale=.3,yscale=.8]
 \draw[thick] (0,0) -- (3,6) --(-3,6) --(0,0);
 \draw (3,6) --(4,10) --(-4,10) --(-3,6);
 \node[below right] at (4,10) {$N$};
 \node[below right] at (3,6) {$M$};
 \node at (-1.35,6.5) (c) [circle, fill=black,scale=.3,label=above:$a$] {};
 \draw[dashed] (0,0) to [out=75,in=-110] (-1.5,6);
 \node[right] at (-1.6,4.5) {\tiny $a^{N}\!{\cap}M$};
\end{tikzpicture}
\caption{The trace of $a$ on $M$}
\end{figure}

The success of the elementary chain construction in building an $\omega_1$-like model relies, of course, on the fact that every countable model of set theory indeed has an elementary top-extension.

\begin{lemma}[Keisler-Morley\cite{KeislerMorley1968:ElementaryExtensionsOfModelsOfSetTheory}]\label{Lemma.KeislerMorleyTopExtensions}
 Every countable model $\<M,{\in^M}>\satisfies\ZFC$ has an elementary top-extension.
\end{lemma}

Let us briefly sketch a folklore proof of this based on definable ultrapowers, as we shall subsequently make use of some of the ideas in the proof. One begins with a countable model $\<M,{\in^M}>\satisfies\ZFC$. The first step is to ensure the global choice axiom, by adding a predicate $C\of M$ such that $\<M,{\in^M},C>$ satisfies $\ZFC(C)$, the version of \ZFC\ that includes instances of the replacement axioms in the expanded language, and also has a $C$-definable well-ordering of the universe. This can be done by the forcing $\Add(\Ord,1)^M$ to add a Cohen class of ordinals $C\of\Ord^M$. Conditions in $\Add(\Ord,1)^M$ are simply the binary ordinal-length sequences $s\in ({}^{{<}\Ord}2)^M$, ordered by extension. Since $M$ is countable, we may find a filter $G\of\Add(\Ord,1)^M$ that meets every dense class $D\of\Add(\Ord,1)^M$ that is definable with parameters over $\<M,{\in^M}>$, and let $C\of\Ord^M$ be the class of which $\Union G$ is the characteristic function. The usual forcing arguments show that $\<M,{\in^M},C>$ satisfies \ZFC\ in the expanded language, using the fact that the forcing $\PP$ is $\kappa$-closed for every $\kappa$ in $M$ and hence adds no new sets. Meanwhile, $\<M,{\in^M},C>$ satisfies global choice, because every set in $M$ is coded by a set of ordinals, and it is dense that any particular set of ordinals shows up as a block in $C$; thus, we may define a global well-order by saying $a<b$ just in case $a$ is coded by a set of ordinals that appears earlier as a block in $C$ than any set of ordinals coding $b$. An isomorphic version of this forcing simply forces to add a bijection $\Ord^M\to M$ explicitly, with conditions consisting of a set-sized piece of such a bijection in $M$; or equivalently, one can generically add a set-like global well-ordering of $M$ by conditions consisting of initial segments of it in $M$.

Let us pause specifically to note that there are continuum many distinct such $C\of\Ord^M$ that we could add to $M$ in this way; in fact there are a perfect set of such $C$. The reason is that we actually have quite a bit of freedom in the construction of the generic filter $G$. Specifically, since $M$ is countable, there are only countably many dense classes $D\of\Add(\Ord,1)^M$ that are definable in $\<M,{\in^M}>$ from parameters, and so we may enumerate them $D_0,D_1,\ldots$, and so on. We build the the generic filter $G$ by selecting a descending sequence of conditions $p_0\geq p_1\geq \cdots$, such that $p_n\in D_n$, and then letting $G$ be the filter generated by these conditions. Notice that at stage $n$, we chose $p_n$ so as to extend the previous condition, but we could also have arbitrarily appended either a $0$ or $1$ on the end of this condition, before choosing $p_{n+1}$. Thus, there is a perfect tree all of whose branches are generic, and different branching choices therefore lead to continuum many different generic filters $G$ and therefore also to continuum many different resulting generic classes $C\of\Ord^M$.

Now, we have a countable model $M[C]=\<M,{\in^M},C>$ satisfying \ZFC\ in the expanded language and also satisfying the global choice principle. If $S$ is the collection of definable classes in this model, allowing parameters, then $\<M,{\in^M},S>$ is a model of the \Godel-Bernays \GBC\ axioms of set theory, and since the construction has the same first-order part, this observation shows that \GBC\ is conservative over \ZFC\ for first-order assertions about sets (an idea attributed to Solovay; see~\cite{Keisler1971:ModelTheoryForInfinitaryLogic}).

The next step of the construction is to find a suitable $M[C]$-ultrafilter $U$ on $\Ord^M$, measuring the definable classes of ordinals in $M[C]$. We find it illuminating to construct $U$ in the forcing style, as a certain kind of $M[C]$-generic filter. Specifically, let $\PP$ be the set of all unbounded $X\of\Ord^M$ that are definable in $M[C]$ from parameters, or in other words, unbounded $X\in S$. We think of this as a forcing notion, where $X$ is stronger than $Y$ if $X\of Y$. Now, let $U\of\PP$ be $M[C]$-generic, in the sense that $U$ contains a member of any dense set $D\of\PP$ that is a definable meta-class in $M[C]$, that is, for which $D=\set{X\in\PP\mid\<M,{\in^M},C,X>\satisfies\varphi(X,\vec a,C)}$ for some first-order formula $\varphi$ and parameter $\vec a\in M$. Since $M$ is countable, there are only countably many such dense meta-classes $D$, and so we may easily construct such an $M[C]$-generic $U$ simply by meeting these dense meta-classes one-by-one. By construction, $U$ does not concentrate on any bounded subset of $\Ord^M$. Note that for any unbounded $X\of\Ord^M$ in $\PP$, the collection of $Y$ such that $Y\of X$ or $Y\of\Ord^M\setminus X$ is dense and definable, and so $U$ thus decides every such unbounded definable set $X\of\Ord^M$ and is therefore an $M[C]$-ultrafilter. In a little while, we shall note a few other properties of $U$ that follow from $M[C]$-genericity.

Meanwhile, we undertake the definable ultrapower construction of $M$ with respect to $U$. For any two functions $F,F':\Ord^M\to M$ that are definable in $M[C]$ from parameters, we define the equivalence relation
 $$F=_UF'\quad\longleftrightarrow\quad\set{\alpha\in\Ord^M\mid F(\alpha)=F'(\alpha)}\in U,$$
and similarly the relation
 $$F\in_UF'\quad\longleftrightarrow\quad\set{\alpha\in\Ord^M\mid F(\alpha)\in F'(\alpha)}\in U,$$
which is well-defined on the $=_U$ equivalence classes $[F]_U$. Let $N$ be the set of such equivalence classes and consider the structure $\<N,{\in^N}>$, where $[F]_U\in^N[F']_U$ if $F\in_UF'$. Using the fact that $M[C]$ has a definable well-ordering of the universe and hence definable Skolem functions, we may establish by the usual induction on formulas that the \Los\ property holds:
$$\<N,{\in^N}>\satisfies\varphi([F]_U)\longleftrightarrow\set{\alpha\in\Ord^M\mid M\satisfies\varphi(F(\alpha))}\in U.$$
In particular, this is a model of \ZFC. Furthermore, the map $a\mapsto[c_a]_U$, where $c_a(\alpha)=a$ is the constant function, is an elementary embedding of $\<M,{\in^N}>$ into $\<N,{\in^N}>$.

Let us now make a few additional observations about the nature of this generic ultrapower. First, we claim that $\<N,{\in^N}>$ is an elementary top-extension of the image of $\<M,{\in^M}>$ in it. This is a consequence of the fact that every bounded definable function is constant on a set in $U$. That is, if $F:\Ord^M\to M$ is definable in $M[C]$ and $X=\set{\alpha\mid F(\alpha)\in V_\beta^M}\in U$, then there is some $a\in M$ such that $F=_U c_a$. One can see this by a simple density argument, since there must be some $a\in V_\beta^M$ such that $X_a=\set{\alpha\mid F(\alpha)=a}$ is unbounded, and it is dense below $X$ to get below some such $X_a$, which will ensure $F=_U c_a$ as desired. It follows that if $[F]_U$ is an element of the ultrapower with rank below that of some $[c_b]_U$ for $b\in M$, then $F(\alpha)\in V_{\text{rank}(b)}^M$ for $U$-almost all $\alpha$, and so $F$ is equal to a constant function $c_a$ for some $a\in V_\beta^M$ on a set in $U$. Thus, every new element of the ultrapower $\<N,{\in^N}>$ is above the rank of the copy of $\<M,{\in^M}>$ inside it. By identifying every $a\in M$ with its image $[c_a]_U$ in $N$, we thereby have a top-extension $\<M,{\in^M}>\sqsubseteq_t \<N,{\in^N}>$.

Further, we claim that $M$ is topless in this extension $N$. To see this, let us first prove that every definable function $F:\Ord^M\to\Ord^M$ is either constant or injective on a set in $U$. If $X\of\Ord^M$ is unbounded, then either $F\restrict X$ is bounded in $\Ord^M$, in which case we can shrink $X$ to some unbounded $X'\of X$ on which $F$ is constant, or $F\restrict X$ has unboundedly many values in $\Ord^M$, in which case we can shrink $X$ to some unbounded $X'\of X$ on which $F$ is injective. So it is dense that the desired property holds. If $F:X\to\Ord^M$ is injective on an unbounded definable set $X\of\Ord^M$, then by shrinking $X$ further, we may assume that $F$ is strictly increasing. Let $\xi_\alpha$ be the $\alpha^{\rm th}$ element of $X$, and let $Y=\set{\xi_{\alpha+1}\mid\alpha\in\Ord^M}$ be the successor elements, which is an unbounded definable subset of $X$. Define $F'(\xi_{\alpha+1})=F(\xi_\alpha)$, which is strictly less than $F(\xi_{\alpha+1})$ since we assumed $F$ was strictly increasing on $X$. Furthermore, $F'$ is also injective and therefore not constant on any unbounded set. So we have proved that it is dense that any definable function $F:\Ord^M\to\Ord^M$ that is not constant on a set in $U$ has a smaller function $F'$ that is also not constant on any set in $U$. Thus, there can be no smallest ordinal in $\<N,{\in^N}>$ above the ordinals corresponding to those in $\<M,{\in^M}>$, and so the extension is topless.

Finally, let us note that $C$ itself arises as the trace on $M$ of an element $c\in N$: $$C=(c^N)\intersect M.$$ Namely, let $F_C(\alpha)=C\intersect\alpha$, which is certainly definable in $M[C]$, and let $c=[F_C]_U$ in the ultrapower $N$. It follows easily that $C=(c)^N\intersect M$, since for $a\in M$ we have $N\satisfies a\in c$ just in case $M\satisfies a\in C$ by the \Los\ property.

Putting all these facts together, we have established the following:

\begin{lemma}\label{Lemma.ContinuumToplessElementaryTopExtensions}
 If $\<M,{\in^M}>\satisfies\ZFC$ is any countable model of set theory, then for continuum many $C\of M$, there is an elementary top extension $\<M,{\in^M}>\elesub_t\<N,{\in^N}>$, in which $M$ is topless and in which $C=(c^N)\intersect M$ arises as the trace on $M$ of an element $c\in N$.
\end{lemma}

We shall use this lemma in our main construction in the next section. It may be interesting for the reader to know that there are $\omega_1$-like models of set theory having no elementary top-extensions, and so one may not omit the countability assumption in lemmas~\ref{Lemma.KeislerMorleyTopExtensions} and \ref{Lemma.ContinuumToplessElementaryTopExtensions} (see~\cite{Kaufmann1983:BluntAndToplessEndExtensionsOfModelsOfSetTheory}). This stands in contrast to the fact that every model of $\PA$, regardless of cardinality, has an elementary end-extension by the MacDowell-Specker theorem~\cite{KossakSchmerl2006:TheStructureOfModelsOfPA}.

Lastly, let us remark that although we found $U$ to be merely $M[C]$-generic and used $M[C]$-definable functions $F$ in the ultrapower construction, a more general approach would be to carry out the construction with respect to an arbitrary countable model of \Godel-Bernays set theory $\<M,\in,S>\satisfies\GBC$ and consider the resulting partial order $\PP$, consisting of conditions that are unbounded $X\of\Ord^M$ with $X\in S$, choosing $U\of\PP$ to be at least $\<M,\in,S>$-generic. The ultrapower in this case would be built out of equivalence classes of functions $F\in S$. One can in principle construct an ultrafilter $U\of\PP$ meeting any desired countable number of dense sets, whether or not these are first-order definable over $\<M,\in,S>$ or second-order definable or what have you. In our presentation above, we used mere $M[C]$-genericity simply because this was convenient and it sufficed for our application.

\section{Incomparable $\omega_1$-like models of set theory}

We shall now prove the first statement of the main theorem, namely, that there can be incomparable $\omega_1$-like models of set theory.

\begin{theorem}\label{Theorem.FamilyOfIncomparableModels}
If $\diamondsuit$ holds and \ZFC\ is consistent, then there is a family $\mathcal C$ of size $2^{\omega_1}$ consisting of pairwise-incomparable $\omega_1$-like models of $\ZFC$, that is, a family for which there is no embedding between any two distinct models in $\mathcal C$.
\end{theorem}

Since the models of any pairwise incomparable family must also of course be pairwise non-isomorphic, it follows that $2^{\omega_1}$ is the largest conceivable size for a family of such pairwise incomparable $\omega_1$-like models. We shall construct the members of the family in a transfinite construction of length $\omega_1$, appealing at each stage to an instance of lemma~\ref{Lemma.KillingOneInstance}, with the specific instance being determined by the $\diamondsuit$-sequence. One should think of lemma~\ref{Lemma.KillingOneInstance} as explaining how permanently to kill off a given embedding $j:M\to N$ of countable models, namely, having extended $N$ to $N^*$, we extend $M$ to $M^*$ in such a way that the embedding $j$ cannot be extended to domain $M^*$, even allowing for further top-extensions of $N^*$ to some $N^{**}$.

\eject
\begin{sublemma}\label{Lemma.KillingOneInstance}
 Suppose that $\<M,{\in^M}>$ and $\<N,{\in^N}>$ are countable models of $\ZFC$ and $j:M\to N$ is an embedding between them. If $N^*$ is any countable proper top-extension of $N$, then there is a countable elementary top-extension $M^*$ of $M$ such that $j$ cannot be extended to an embedding of $M^*$ into to any top-extension of $N^*$.

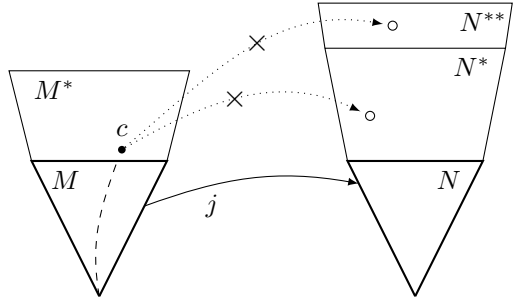
$$
\begin{tikzpicture}[scale=.3,>=latex]
 \draw[thick] (0,0) -- (3,6) --(-3,6) --(0,0);
 \draw (3,6) --(4,10) --(-4,10) --(-3,6);
 \node[below] at (-1.5,6) {$M$};
 \node[below] at (-2,10) {$M^*$};
 \draw[thick] (14,0) -- (17,6) --(11,6) --(14,0);
 \draw (17,6) -- (18,11) --(10,11) --(11,6);
 \draw[thin] (18,11) --(18.3,13) --(9.7,13) --(10,11);
 \node[below] at (15.5,6)  {$N$};
 \node[below] at (16.5,11) {$N^*$};
 \node[below] at (17,13) {$N^{**}$};
 \draw[->] (2,4) to [out=20, in=170] (11.5,5);
 \node at (5,4) {$j$};
 \node at (1,6.5) (c) [circle, fill=black,scale=.3,label=above:$c$] {};
 \node at (12,8) (jc) {};
 \node at (13,12) (jc2) {};
 \draw (jc)  circle (.2);
 \draw (jc2) circle (.2);
 \draw[->,dotted] (c) to [out=30, in=155] (jc);
 \draw[->,dotted] (c) to [out=40, in=170] (jc2);
 \node at (6,8.75) {\Large $\times$};
 \node at (7,11.25) {\Large $\times$};
 \draw[dashed] (0,0) to [out=100,in=-110] (.8,6);
\end{tikzpicture}
$$
\end{sublemma}

\begin{proof}
Suppose that $j:M\to N$ is an embedding of the countable models of set theory $\<M,{\in^M}>$ and $\<N,{\in^N}>$, and that $N\of_t N^*$ is a given top-extension (not necessarily elementary). For each $b\in N^*$, let $X_b=\set{a\in M\mid j(a)\in^{N^*}b}$, which is the same as the pre-image $j^{-1}(b^{N^*}\intersect N)$ of the trace of $b$ on $N$. Since $N^*$ is countable, there are only countably many such subsets $X_b$ of $M$. Thus, by lemma~\ref{Lemma.ContinuumToplessElementaryTopExtensions}, there is an elementary top-extension $M^*$ of $M$ with an element $c\in M^*$ whose trace on $M$, that is, $c^{M^*}\intersect M$, is not $X_b$ for any $b\in N^*$. It follows that $j$ has no extension to an embedding $j:M^*\to N^{**}$ to any top-extension $N^{**}$ of $N$, because there will be no suitable target for $c$. Specifically, for any such extension $j^*$ of $j$ consider $b'=j^*(c)$, and let $\alpha$ be an ordinal of $N^{**}$ that is above $N$ and below $N^*$, and let $b=b'\intersect V_\alpha$, so that $b\in N^*$ by the top extension property. But for $a\in M$ we have $a\in c\longleftrightarrow j^*(a)=j(a)\in j^*(c)\longleftrightarrow j(a)\in b$, since $j(a)$ is in $N$ and thus in $b'$ if and only if it is in $b$. This shows $c^{M^*}\intersect M=X_b$ after all, contrary to our choice of $c$.
\end{proof}

Thus, having extended $N$ to $N^*$, we may permanently kill off the embedding $j:M\to N$ by extending $M$ to $M^*$, as after this there can now be no suitable target for the object $c$.

\begin{proof}[Proof of theorem~\ref{Theorem.FamilyOfIncomparableModels}] Assume $\diamondsuit$ holds, which means that there is an $\omega_1$-sequence $\<A_\alpha\mid\alpha<\omega_1>$, fixed for the rest of the argument, such that $A_\alpha\of\alpha$ and for every $A\of\omega_1$, the set $\set{\alpha\mid A\intersect\alpha=A_\alpha}$ is stationary. We shall now assign to each countable-ordinal binary sequence $s\in{}^{\lt\omega_1}2$ a countable model $M_s=\<M_s,{\in^{M_s}}>\satisfies\ZFC$, in such a way that extending a sequence means elementarily top-extending the model, $s\of t\implies M_s\elesub_t M_t$. Further, we shall ensure that the construction is continuous at limit stages in the sense that $M_s=\Union_{\alpha<\lambda}M_{s\restrict\alpha}$ is the union of the corresponding elementary chain whenever $s$ has limit length $\lambda$. Similarly, at the very top, we define for each uncountable branch $S\in {}^{\omega_1}2$ the model $M_S$ as the union of the corresponding continuous elementary chain $M_S=\Union_{\alpha<\omega_1} M_{S\restrict\alpha}$, determined by the branch $S$. Thus, we have really built a continuous tree of models $M_s$, and our final family will consist precisely of the models $M_S$ arising as the branches through this tree. It will be convenient for us that the underlying set of each $M_s$ is a countable ordinal.

Our construction proceeds in $\omega_1$ many stages, defining $M_s$ by recursion on the length of $s$, so that $M_s$ for $s\in {}^\alpha 2$ will become defined at stage $\alpha$. We may begin at stage $0$ at the bottom with any desired countable model $M_{\emptyset}$ of \ZFC, with underlying set $\omega$. At most stages of the construction, including every finite stage and every stage that is neither a limit ordinal nor a successor to a limit ordinal, if $M_s$ has just been defined, then we will let $M_{s\concat 0}$ and $M_{s\concat 1}$ be arbitrary countable elementary top-extensions of $M_s$, using some larger countable ordinal as the underlying set. The interesting part of the construction occurs at a limit ordinal $\lambda$, where $M_s$ is defined for all $s\in {}^{\lt\lambda}2$. By continuity, we define $M_s$ for $s\in{}^\lambda 2$ as the union $M_s=\Union_{\alpha<\lambda} M_{s\restrict\alpha}$. Now, for the critical step, we consult the set $A_\lambda$ appearing in the diamond sequence and interpret it in some canonical manner as coding two elements $\sbar,\tbar\in{}^\lambda 2$ and a subset $j\of\lambda\times\lambda$. If it happens by some miracle that the underlying sets of $M_\sbar$ and $M_\tbar$ are both equal to $\lambda$ and furthermore that $j:M_\sbar\to M_\tbar$ is an embedding, then we define $M_{\tbar\concat 0}=M_{\tbar\concat 1}$ to be an arbitrary proper countable elementary top-extension of $M_\tbar$, and we define $M_{\sbar\concat 0}=M_{\sbar\concat 1}$ to be the elementary extension $M^*$ of lemma~\ref{Lemma.KillingOneInstance}, which ensures that this $j$ will not extend further to an embedding of these taller models (taking copies of these structures to have underlying set as a countable ordinal). If the miracle situation does not occur, then as we explained, the models are to be extended one more step in an arbitrary elementary top-extension manner. This completes the definition of $M_s$ for every $s\in{}^{\lt\omega_1}2$ and hence also of $M_S$ for $S\in {}^{\omega_1}2$.

By construction, each $M_S$ is the union of an elementary $\omega_1$-chain of proper top extensions of $M_\emptyset$, and hence is an $\omega_1$-like model of \ZFC. But we claim that there can be no embedding between distinct such models. To see this, suppose that $j:M_S\to M_T$ is an embedding, where $S\neq T$ in ${}^{\omega_1}2$. Let $A\of\omega_1$ code the three objects $S, T$ and $j$, using the same canonical coding method used in the construction. It follows by the $\diamondsuit$ principle that $A_\lambda=A\intersect\lambda$ for a stationary set of $\lambda$. Since the underlying set of $M_{S\restrict\lambda}$ and $M_{S\restrict\lambda}$ are both equal to $\lambda$ for a club of $\lambda$, and furthermore $j\image\lambda\of\lambda$ also occurs on a club of limit ordinals $\lambda$, there must be a stage $\lambda$ in the construction where the set $A_\lambda$ is exactly giving us $S\restrict\lambda$, $T\restrict\lambda$ and $j\restrict\lambda$, where the models $M_{S\restrict\lambda}$ and $M_{T\restrict\lambda}$ both have underlying set $\lambda$ and $j\restrict\lambda$ is an embedding between them. In this (miraculous) case, we specifically ensured that $M_{S\restrict\lambda+1}$ was chosen in such a way that $j\restrict M_{S\restrict\lambda}$ had no extension to an embedding of $M_{S\restrict\lambda+1}$ into any further top-extension of $M_{T\restrict\lambda+1}$. This contradicts our assumption that $j:M_S\to M_T$ is an embedding, since $j\restrict M_{S\restrict \lambda+1}$ would be such an embedding. So the family of models $\set{M_S\mid S\in{}^{\omega_1} 2}$ must admit no such embeddings after all, just as we claimed.
\end{proof}

Since the choice of $M_\emptyset$ was arbitrary, the proof actually shows that for any consistent theory $T$ extending $\ZFC$, there are $2^{\omega_1}$ many $\omega_1$-like pairwise non-embeddable models of $T$. For example, all the models will satisfy $V=L$, if $M_\emptyset$ does.

Let us now consider the question of whether the models $M_S$ in the family $\mathcal C$ constructed in theorem~\ref{Theorem.FamilyOfIncomparableModels} also serve as $\omega_1$-like counterexamples to the assertion that one model of set theory embeds into another, if the ordinals of the first model order-embed into the ordinals of the second.

\begin{theorem}\label{Theorem.FamilyOfIncomparableModelsSameOrdinals}
Under the hypothesis of theorem~\ref{Theorem.FamilyOfIncomparableModels}, the models in the family $\mathcal C$ can be constructed so that their ordinals all order-embed into one another, and furthermore, so that their ordinals are universal for all $\omega_1$-like linear orders. If \ZFC\ has an $\omega$-model, then the models in family $\mathcal C$ can be constructed so as all to have precisely the same ordinals.
\end{theorem}

\begin{proof}
To prove this, we shall simply pay a little closer attention to the ordinals of the models $M_s$ in the construction of theorem~\ref{Theorem.FamilyOfIncomparableModels}. All the models $M_s$ in that construction have the model $M_\emptyset$ at the root as a common initial segment, and we may assume without loss that $M_\emptyset$ is nonstandard. It follows that $\Ord^{M_\emptyset}$ contains a copy of the countable dense linear order $\Q$, and since the ordinals are closed under addition, we will find copies of this $\Q$ unboundedly often in the ordinals $\Ord^{M_s}$ of each of the models $M_s$ that we construct. Consequently, the ordinals $\Ord^{M_S}$ of the models $M_S$ constructed at the top, where $S\in{}^{\omega_1}2$, will be an $\omega_1$-like linear order containing unboundedly many non-overlapping copies of $\Q$. In particular, $\Ord^{M_S}$ contains the long rational line $\Q\cdot\omega_1$ as a suborder. This order is easily seen to be universal for all $\omega_1$-like linear orders, since if $\<A,<>$ is any $\omega_1$-like linear order, realized as the union $A=\Union_{\alpha<\omega_1}A_\alpha$ of a continuous chain of countable initial segments, then we may map $A_0$ into the first copy of $\Q$ and map each difference set $A_{\alpha+1}-A_\alpha$ order-preservingly into a fresh copy of $\Q$ above what came below, thereby embedding all of $A$ into $\Q\cdot\omega_1$. So the ordinals $\Ord^{M_S}$ of every model $M_S\in \mathcal C$ are universal in this way and in particular, they all order-embed into one another.

By making the slightly stronger assumption that \ZFC\ has an $\omega$-model, we may ensure that all the models $M_S$ have precisely the same ordinals. Namely, begin by taking $M_\emptyset$ to be a countable $\omega$-standard nonstandard model of \ZFC. It follows by a result of Friedman~\cite{Friedman1973:CountableModelsOfSetTheories} that $\Ord^{M_\emptyset}$ has order type $\lambda+\lambda\cdot\Q$ for some admissible ordinal $\lambda$, which is simply the well-founded part of $\Ord^{M_\emptyset}$. Let us also assume that in the construction of the models, whenever we build a top-extension $M_{s\concat i}$ over $M_s$, we always do so by means of the construction described before lemma~\ref{Lemma.ContinuumToplessElementaryTopExtensions}, which means in particular that the extension $M_s\prec_t M_{s\concat i}$ is topless. Since the well-founded part of the ordinals of these models is still $\lambda$, the well-founded part of $M_\emptyset$, it follows that the additional ordinals of $\Ord^{M_{s\concat i}}$ on top of $\Ord^{M_s}$ have order-type precisely $\lambda\cdot\Q$. The final models $M_S$ at the top, therefore, arise by a process that places another $\Q$ copies of $\lambda$ on top of the previous model, performing this $\omega_1$ many times. Thus, the ordinals $\Ord^{M_S}$ of any of the models $M_S$ in $\mathcal C$ will have order type $\lambda+(\lambda\cdot\Q)\cdot\omega_1$. In particular, the ordinals of all these models are order-isomorphic and by replacing with an isomorphic copy we may assume that all the models $M_S$ have precisely the same ordinals.
\end{proof}

Thus, we have now proved statements (1) and (2) of the main theorem stated in the introduction. Let us turn briefly to statement (3), which can be established by a similar argument.

\begin{theorem}\label{Theorem.DiamondImpliesOmega1LikeZFCnotembedPA}
If $\diamondsuit$ holds and \ZFC\ is consistent, then there is an $\omega_1$-like model $M\models\ZFC$ and an $\omega_1$-like model $N\models\PA$ such that $M$ does not embed into $\<\HF,{\in}>^N$.
\end{theorem}

\begin{proof}
To construct $M$ and $N$, we shall carry out a simplified version of the construction of the proof of theorem~\ref{Theorem.FamilyOfIncomparableModels}. First, we note that the proof of lemma~\ref{Lemma.KillingOneInstance} also establishes an analogous fact for models of finite set theory; we omit the proof.

\begin{sublemma}\label{Lemma.KillingOneInstanceFiniteZFC}
 If $j:M\to \HF^N$ is an embedding of a model of set theory $M\satisfies\ZFC$ into the hereditary finite sets $\HF^N$ of a countable model of arithmetic $N\satisfies\PA$ and $N\prec_e N^*$ is any proper elementary end-extension, then there is an elementary top-extension $M\elesub_t M^*$ such that $j$ does not extend to an embedding $j:M^*\to \HF^{N^{**}}$ for any further end-extension $N^*\subseteq_e N^{**}$ of $N^*$.
\end{sublemma}

\noindent Given this lemma, we shall prove the theorem by building the models $M=\Union_{\alpha<\omega_1}M_\alpha$ and $N=\Union_{\alpha<\omega_1}N_\alpha$ as the unions of corresponding elementary chains of countable models $M_\alpha$  and $N_\alpha$. We may begin with any two countable models $M_0\satisfies\ZFC$ and $N_0\satisfies\PA$. At most stages, including every finite stage and every stage that is neither a limit ordinal nor a successor to a limit ordinal, we let $M_{\alpha+1}$ be an arbitrary proper elementary top-extensions of $M_\alpha$ and let $N_{\alpha+1}$ be an arbitrary proper elementary end-extension of $N_\alpha$, using some countable ordinal as an underlying set. At a limit stage $\lambda$, we first define $M_\lambda=\Union_{\alpha<\lambda}M_\alpha$ and $N_\lambda=\Union_{\alpha<\lambda}N_\alpha$ to be the union of the corresponding elementary chains of models constructed so far. Next, the critical step, we consult the $\diamondsuit$-sequence, interpreting it as a set $j\of\lambda\times\lambda$, and if it happens (by some miracle) that the underlying sets of $M_\lambda$ and $N_\lambda$ are both equal to $\lambda$ and $j:M_\lambda\to \HF^{N_\lambda}$ is an embedding, then we first properly elementarily end-extend $N_\lambda$ to $N_{\lambda+1}$. It follows that $\HF^{N_{\lambda+1}}$ is an elementary top-extension of $\HF^{N_\lambda}$, and so by lemma~\ref{Lemma.KillingOneInstanceFiniteZFC} we may extend $M_\lambda$ to $M_{\lambda+1}$ in such a way that prevents $j$ from extending to this larger domain. It follows as before that there can be no embedding $j:M\to\HF^N$ ultimately, because initial segments of this embedding will have been prevented from extending, just as in the proof of theorem~\ref{Theorem.FamilyOfIncomparableModels}.
\end{proof}

Let us turn now to the final statement of the main theorem, asserting that it is consistent relative to a Mahlo cardinal that there is a transitive $\omega_1$-like model $M\models\ZFC$ that does not embed into its constructible universe $L^M$. A cardinal $\kappa$ is {\df Mahlo}, if it is inaccessible and the regular cardinals below $\kappa$ form a stationary subset of $\kappa$. Note that in order to prove the statement, at least some large cardinal assumption will be necessary (as well as $V\neq L$), since there is an $\omega_1$-like transitive model of \ZFC\ just in case $L_{\omega_1}\satisfies\ZFC$, and this is equivalent to the assertion that $\omega_1$ is inaccessible in $L$. This hypothesis is equiconsistent with the existence of an inaccessible cardinal, since any inaccessible cardinal $\kappa$ can become the $\omega_1$ of a forcing extension, by forcing for example with the \Levy\ collapse of $\kappa$.

\begin{theorem}\label{th:incomparableTransitive}
If $\kappa$ is Mahlo, then there is a forcing extension of $L$ in which $\kappa$ becomes $\omega_1$ and where there is a transitive $\omega_1$-like model $M\satisfies\ZFC$ that does not embed into its own constructible universe $L^M$.
\end{theorem}

\begin{proof}
If $\kappa$ is Mahlo, then this is absolute down to $L$, and so we may assume without loss that $V=L$ in our ground model. The forcing will have two large steps: the first step will create the desired model $M=L_\kappa[G]$; and the second step will be the \Levy\ collapse of $\kappa$, ensuring that this model becomes $\omega_1$-like in the final extension $L[G][H]$.

To begin, let $\PP=\Pi_{\gamma<\kappa}\Add(\gamma,1)$ be the Easton-support product of the forcing to add a Cohen subset to every regular cardinal $\gamma$ below $\kappa$ (so the product here is indexed by the cardinals $\gamma$ below $\kappa$). Since $\kappa$ is Mahlo, this is $\kappa$-c.c.~and the usual Easton factor arguments show that all cardinals and cofinalities are preserved. Second, let $\Q=\Coll(\omega,{<}\kappa)$ be the \Levy\ collapse of $\kappa$, that is, the finite-support product $\Q=\Pi_{\gamma<\kappa}\Coll(\omega,\gamma)$, which collapses every cardinal $\gamma$ below $\kappa$ to $\omega$. For any $\delta<\kappa$, let $\PP_\delta=\PP\restrict\delta=\Pi_{\gamma<\delta}\Add(\gamma,1)$ and $\Q_\delta=\Q\restrict\delta=\Pi_{\gamma<\delta}\Coll(\omega,\gamma)$ be the corresponding initial segments of the forcing $\PP$ and $\Q$. Suppose that $G\times H\subseteq\PP\times\Q$ is $V$-generic, and let $G_\delta$ and $H_\delta$ be the restrictions of $G$ and $H$ to the initial segments $\PP_\delta$ and $\Q_\delta$.

Since $\kappa$ was Mahlo in $L$, we have $L_\kappa\satisfies\ZFC$. From the perspective of $L_\kappa$, the forcing $\PP$ is progressively closed class forcing, and so $L_\kappa[G]\satisfies\ZFC$ as well. The \Levy\ collapse $\Q$ is $\kappa$-c.c. over $L[G]$ and forces $\kappa=\omega_1^{L[G][H]}$. It follows that $L_\kappa[G]$ is $\omega_1$-like in $L[G][H]$.

We claim that there is no embedding $j:L_\kappa[G]\to L_\kappa$ in $L[G][H]$. Suppose toward contradiction that $j$ is such an embedding. Fix a $\PP\times\Q$-name $\sigma$ such that $\sigma_{G\times H}=j$ and a condition $(p,q)\in G\times H$ forcing that $\sigma$ is an embedding from $L_\kappa[G]$ to $L_\kappa$. Let us say that $\sigma\restrict\gamma$ is \emph{determined} by stage $\gamma$ if for each $\xi<\gamma$ there is a maximal antichain below $(p,q)$ in $\PP\times\Q$, with support contained in $\gamma$ in each factor, such that every condition in the antichain decides $\sigma(\check\xi)$. It follows in this case that $\sigma_{G\times H}\restrict\gamma$ is already in $V[G_\gamma][H_\gamma]$.

Since $\PP\times\Q$ is $\kappa$-c.c., it is easy to see by a simple closing-off argument that there is a club subset $C\of\kappa$ such that $\sigma\restrict\gamma$ is determined by stage $\gamma$ for all $\gamma\in C$. Since $\kappa$ is Mahlo, there is such a $\delta\in C$ that is inaccessible, and in particular, $\delta$ is a stage of nontrivial forcing in $\PP$. Let $A\of\delta$ be the Cohen set added by $\Add(\delta,1)$ in the forcing $\PP$ at coordinate $\delta$. Thus, $A$ is $L[G_\delta][H_\delta]$-generic. Since $\sigma\restrict\delta$ is determined by stage $\delta$, it follows that $j\restrict\delta\in L[G_\delta][H_\delta]$. By assumption, $j(A)\in L$. Since $j$ is an embedding, we have $\alpha\in A\longleftrightarrow j(\alpha)\in j(A)$, and from this it follows that $A\in L[G_\delta][H_\delta]$, contrary to genericity.
\end{proof}

Note that if we omit the second part of the forcing, what we have is the $\kappa$-like model $L_\kappa[G]$ in $L[G]$, which in $L[G]$ does not embed into its constructible universe $L_\kappa=L^{L_\kappa[G]}$. The only purpose of the \Levy\ collapse was to enable the phenomenon to occur with an $\omega_1$-like model. A similar argument shows that if \Ord\ is Mahlo in $V$, then in the corresponding forcing extension $V[G]$, where we undertake the Easton-support iteration to add a Cohen subset to each regular cardinal, there is no class $j$ that is an embedding $j:V[G]\to L$.

\section{Questions}

Several questions surrounding the subject of this article remain open. First, we wonder whether we really need the $\diamondsuit$ hypothesis in the main theorem.

\begin{question}
Can we eliminate the $\diamondsuit$ assumption in the main theorem? Specifically, is the existence of embedding-incomparable $\omega_1$-like models of \ZFC\ provable in \ZFC\ from the consistency of \ZFC?
\end{question}

Analogous situations have often arisen in the context of models of arithmetic, where the first example of an $\omega_1$-like model with certain features is constructed under the $\diamondsuit$ hypothesis, but subsequent more refined arguments eliminate the need for that assumption  (see~\cite{Shelah1978:ModelsWithSecondOrderPropertiesIITreesWithNoUndefinedBranches} for the general $\diamondsuit$ elimination technique). So we are accustomed in the subject to positive resolutions of similar instances of this question. Further evidence for a positive answer may be the fact that Kossak~\cite{Kossak1985:RecursivelySaturatedOmega1LikeModels} showed, with no $\diamondsuit$ assumption, that there is pair of elementarily equivalent $\omega_1$-models of $\PA$ with the same standard system, such that neither embeds into the other. His proof technology uses minimal types, conservative extensions, and most importantly, the fact that embeddings of $\PA$-models are automatically $\Delta_0$-elementary. This last point, as we have noted, is not true for embeddings of models of set theory and suggests that Kossak's proof for models of $\PA$ will not generalize directly to the $\ZFC$ context.

\begin{question}
 Is it consistent relative to an inaccessible cardinal that there is an $\omega_1$-like transitive model $M\models\ZFC$ having no embedding $j:M\to L^M$ into its own constructible universe?
\end{question}

In other words, can the Mahlo cardinal hypothesis of theorem~\ref{th:incomparableTransitive} be reduced to merely an inaccessible cardinal? As we noted in the discussion before theorem~\ref{th:incomparableTransitive}, the existence of an $\omega_1$-like transitive model of \ZFC\ is equivalent to the assertion that $\omega_1$ is inaccessible in $L$, and so one needs at least an inaccessible cardinal. If one drops the transitivity requirement, then it is conceivable that an argument could proceed merely from $\Con(\ZFC)$.

\begin{question}
 Is it consistent relative to $\Con(\ZFC)$ that there is a (possibly nonstandard) $\omega_1$-like model $M\models\ZFC$ having no embedding $j:M\to L^M$ into its own constructible universe?
\end{question}

We have as yet no nonstandard instances of such a model, from any hypothesis. The model constructed in theorem~\ref{th:incomparableTransitive} was standard, and used the hypothesis of a Mahlo cardinal. We expect that one may be able to construct nonstandard instances from much weaker hypotheses.

Finally, we have some questions concerning the absoluteness of the nonexistence of embeddings between $\omega_1$-like models.

\begin{question}
Is it consistent that there are $\omega_1$-like models $M$ and $N$ of $\ZFC$ such that neither embeds into the other, yet there are $\omega_1$-preserving forcing notions adding embeddings in either direction? Conversely, is it consistent to have such incompatible models with the property that in any outer model that sees an embedding in either direction, $\omega_1$ is collapsed?
\end{question}

\bibliographystyle{alpha}
\bibliography{HamkinsBiblio,MathBiblio}

\end{document}